\documentclass[11pt]{article}
\usepackage{amsmath, amssymb, theorem, latexsym, epsfig}
\numberwithin{equation}{section}

\theoremstyle{plain}
\theorembodyfont{\itshape}
\newtheorem{theorem}{Theorem}[section]
\newtheorem{proposition}[theorem]{Proposition}
\newtheorem{lemma}[theorem]{Lemma}
\newtheorem{corollary}[theorem]{Corollary}
                                                                                
\theorembodyfont{\rmfamily}
\newtheorem{definition}[theorem]{Definition}

\newtheorem{remark}[theorem]{Remark}
\newtheorem{convention}[theorem]{Convention}
\newenvironment{proof}{{\noindent \textbf{Proof}\,\,}}{\hspace*{\fill}$\Box$\medskip}

\title{On commuting billiards in higher-dimensional spaces of constant curvature}
\author{Alexey Glutsyuk
\thanks{ CNRS, France (UMR 5669 (UMPA, ENS de Lyon) and UMI 2615 (Interdisciplinary Scientific Center J.-V.Poncelet)), 
Lyon, France. 
E-mail:
aglutsyu@ens-lyon.fr}
\thanks{National Research University Higher School of Economics (HSE), Moscow, Russia}
 \thanks{Supported by part by RFBR grants  16-01-00748 and 16-01-00766}}
\begin{document}
\maketitle
\def\ii{\mathbb I}
\def\mca{\mathcal A}
  \def\diag{\operatorname{diag}}
 \def\la{\lambda}
\def\rr{\mathbb R}
\def\rp{\mathbb{RP}}
\def\dist{\operatorname{dist}}
\def\cc{\mathbb C}
\def\la{\lambda}
\def\hh{\mathbb H}
\def\hn{\hh_d}
\def\nn{\mathbb N}
 \def\mct{\mathcal T}
 \def\Int{\operatorname{Int}}
 \def\mcb{\mathcal B}
 \def\zz{\mathbb Z}
 \def\rp{\mathbb{RP}}
 \def\mcb{\mathcal B}
 \def\mcl{\mathcal L}
\def\deg{\operatorname{deg}}
\def\mod{\operatorname{mod}}
\def\ker{\operatorname{Ker}}
\def\mcd{\mathcal D}
\def\mcu{\mathcal U}
\def\mcv{\mathcal V}
\def\wt#1{\widetilde#1}
\def\La{\Lambda}

\begin{abstract}
We consider two nested   billiards in $\rr^d$, $d\geq3$, with $C^2$-smooth strictly 
convex boundaries. We prove that if the corresponding actions by reflections on the space of oriented lines commute, then the billiards are confocal ellipsoids. This together with 
the previous analogous result of the author in two dimensions solves completely the 
Commuting Billiard Conjecture due to Sergei Tabachnikov. The main result is deduced 
from the classical theorem due to Marcel Berger saying that in higher dimensions 
only quadrics may have caustics. We also prove versions of
 Berger's theorem and the main result  for billiards in spaces of constant curvature: 
 space forms. 
\end{abstract}
\tableofcontents

\section{Introduction}
\subsection{Main result}
Let $\Omega_a\Subset\Omega_b\subset\rr^d$ be two nested bounded 
domains with smooth strictly convex boundaries  $a=\partial\Omega_a$ and $b=\partial\Omega_b$.  Consider the corresponding 
billiard transformations $\sigma_a$, $\sigma_b$ acting  on the space of oriented lines in space by reflection as follows. Each $\sigma_g$, $g=a,b$, acts as identity on the lines disjoint from $g$. For each oriented line $l$ intersecting $g$  we take  its 
last intersection point $x$ with $g$ in the sense of orientation:  the orienting arrow of the line $l$ at $x$ 
is directed outside $\Omega_g$.  The image $\sigma_g(l)$ is the line obtained by reflection of the line $l$ from the hyperplane $T_xg$: the angle of incidence equals 
the angle of reflection. The line $\sigma_g(l)$ 
 is oriented by a tangent  vector at $x$ directed inside $\Omega_g$.  This is a continuous mapping that is smooth on the space of lines intersecting $g$ transversely. 
 
 \begin{remark} The above action can be defined for a convex billiard in any Riemannian manifold; the billiard reflection acts on the space of oriented geodesics.
 \end{remark}

Recall, see, e.g., \cite{berger, tab}, that a pencil of {\it confocal quadrics} in a Euclidean 
space $\rr^d$ is a one-dimensional family of quadrics defined in some orthogonal coordinates $(x_1,\dots,x_d)$ by equations 
$$\sum_{j=1}^d\frac{x_j^2}{a_j^2+\la}=1; \ \ \ a_j\in\rr \text{ are fixed; } \ \la\in\rr \text{ is the parameter}.$$

It is known that {\it any two confocal elliptic or ellipsoidal billiards}  commute 
\cite[p.59, corollary 4.6]{tab}, \cite[p.58]{tabcom}. Sergei Tabachnikov stated 
the conjecture affirming the converse: any two commuting nested convex 
billiards are confocal ellipses 
(ellipsoids) \cite[p.58]{tabcom}. In two dimensions his conjecture was proved 
by the author of the present paper in \cite[theorem 5.21, p.231]{anal} for piecewise $C^4$-smooth boundaries. 
 Here we prove it in higher dimensions in $\rr^d$ and in spaces of constant curvature: 
 space forms.

\begin{theorem} \label{tab} Let two  nested  strictly convex  $C^2$-smooth closed 
hypersurfaces in $\rr^d$, 
 $d\geq3$,  be such that the corresponding billiard transformations commute. Then they are  confocal  ellipsoids.
\end{theorem} 

To extend Theorem \ref{tab} to spaces of constant curvature, let us recall the notions 
of space forms and (confocal) quadrics in them. 

\begin{definition} A {\it space form} is a complete connected Riemannian manifold 
of constant curvature.
\end{definition}

\begin{remark} We will deal only with simply connected space forms. 
It is well-known that they  
are the Euclidean space $\rr^d$, the unit sphere $S^d\subset\rr^{d+1}$ in the Euclidean space and the hyperbolic space $\hn$ (up to normalization of the metric 
 by constant scalar factor, which changes neither geodesics, nor reflections). 
 It is known that the hyperbolic space $\hn$ 
admits a standard model in the Minkovski space $\rr^{d+1}$. Finally, 
each space form $\Sigma$ we will be dealing with is 
 realized as an appropriate hypersurface in the space $\rr^{d+1}$ with coordinates 
 $x=(x_0,\dots,x_d)$ 
equipped with a suitable quadratic form
$$<Gx,x>, \  \ G \text{ is a symmetric } \  (d+1)\times(d+1)-
\text{matrix}:$$

Euclidean case: $G=\diag(0,1,\dots,1)$, $\Sigma=\rr^d=\{ x_0=1\}$.

Spherical case: $G=Id$, $\Sigma=S^d=\{<Gx,x>=1\}$, $<Gx,x>=\sum_jx_j^2$. 

Hyperbolic case: $G=\diag(-1,1,\dots,1)$, $\Sigma=\hn=\{<Gx,x>=-1\}\cap\{ x_0>0\}$. 

The metric on each hypersurface $\Sigma$ is the restriction to $T\Sigma$ of the 
quadratic form $<Gx,x>$ on the ambient space. It is well-known that the geodesics 
on $\Sigma$ are  its intersections with two-dimensional vector subspaces in $\rr^{d+1}$. 
Completely geodesic $k$-dimensional submanifolds in $\Sigma$ are its intersections 
with $(k+1)$-dimensional vector subspaces in $\rr^{d+1}$. 
\end{remark}

\begin{definition} \cite[p. 84]{veselov2} A {\it quadric} in $\Sigma$ is a hypersurface 
$$S=\Sigma\cap\{<Qx,x>=0\}, \ Q \text{ is a symmetric matrix}.$$
The  {\it pencil of confocal quadrics} associated to a symmetric matrix $Q$ 
 is the family of quadrics 
 $$S_\la=\Sigma\cap\{<Q_\la x,x>=0\}, \ Q_{\la}=(Q-\la G)^{-1}, \ \la\in\rr.$$
 \end{definition}
 
 \begin{definition}  \label{stconv} A germ of $C^2$-smooth 
 hypersurface $S$ in a space form $\Sigma$ at a point $p$ is {\it strictly convex}, if it has 
 quadratic tangency with its tangent completely geodesic hypersurface $\Gamma_p$: 
 there exists a constant $C>0$ such that for every $q\in S$ close to $p$ one has 
 $$\dist(q,\Gamma_p)>C||q-p||^2; \text{ here } ||q-p||=\dist(q,p).$$ 
 \end{definition}

\begin{theorem} \label{tab2} Let $d\geq3$, and let $\Sigma$ be a simply connected $d$-dimensional space form: either $\rr^d$, or the unit sphere, or the hyperbolic space. 
Let two  nested  strictly convex  $C^2$-smooth closed 
hypersurfaces in $\Sigma$  be such that the corresponding billiard transformations commute. Then they are confocal quadrics. 
\end{theorem} 

 Theorem \ref{tab} follows from Theorem \ref{tab2}.
 
Theorem \ref{tab} can be deduced from a classical theorem due to Marcel Berger 
\cite{berger} 
concerning billiards in $\rr^d$, $d\geq3$, which states that only billiards bounded by 
quadrics may have caustics (see Definition \ref{defca} 
for the notion of caustic), and the caustics are their confocal quadrics.  
To prove Theorem \ref{tab2} in full generality, we extend Berger's theorem to the case of  billiards in space forms (Theorem \ref{berg} stated in Subsection 1.3 and proved in 
Section 2) and then deduce Theorem  \ref{tab2} 
in Section 3. A local version of Theorem \ref{tab2} will be proved in Section 4. 
In Section 5 we present some open problems. 

\subsection{Historical remarks}
Commuting billiards are closely related to problems of 
classification of integrable billiards, see \cite{tabcom}.  It is known that 
elliptic and ellipsoidal billiards are  integrable, see \cite[proposition 4]{veselov},  \cite[chapter 4]{tab}, and this also holds for non-Euclidean ellipsoids in sphere and 
in the Lobachensky (hyperbolic) space of any dimension, see \cite[the corollary on p. 95]{veselov2}. The famous Birkhoff Conjecture states that in two dimensions the converse is true. Namely, it deals with the so-called 
{\it Birkhoff caustic-integrable} convex planar billiards with smooth boundary, that is, 
billiards for which 
there exists a foliation by closed caustics in an interior neighborhood of 
the boundary. It states that the only Birkhoff caustic-integrable billiards are ellipses.  
Birkhoff Conjecture was first stated in print in Poritsky's paper \cite{poritsky}, who  proved it in loc. cit. 
under the additional assumption that the billiard in each closed caustic near the boundary has the same closed caustics, as the initial billiard. Poritsky's assumption implies that 
{\it the initial billiard map commutes with the billiard in any closed caustic;} this follows by 
the arguments presented in \cite[section 4, pp.58--59]{tab}. 
One of the most famous results 
on Birkhoff Conjecture is a theorem of M.Bialy \cite{bialy}, who proved that if the phase cylinder of the billiard map is  foliated (almost everywhere) by non-contractible closed curves which are invariant under the billiard map, then the boundary is a circle. 
In \cite{bialy1} he proved the same result for billiards on surfaces of non-zero 
constant curvature. 
 A local version of Birkhoff Conjecture,  for integrable deformations of ellipses was recently solved in 
\cite{kavila, kalsor}. Recent solution of 
its algebraic version (stated and partially studied in \cite{bolotin2}) is a result of papers \cite{bm, bm2, hess, gl}. 
For a historical survey of Birkhoff Conjecture 
see  \cite[section 5, p.95]{tab}, the recent surveys \cite{BM18, KS18}, the 
papers \cite{kalsor, hess} and references therein. 
Dynamics in  billiards  in two and higher 
dimensions with piecewise smooth boundaries consisting of confocal quadrics 
was studied in \cite{drag}.

\subsection{Berger's theorem and its extension to 
billiards in space forms} 

\begin{definition} \label{defca} 
Let $a$, $b$ be two nested strictly convex closed hypersurfaces 
in a Riemannian manifold $E$: the hypersurface $b$ bounds a relatively compact 
domain in $E$ whose interior  contains $a$. 
We say that $a$ is a {\it caustic} for the hypersurface $b$, if the image of each oriented 
geodesic tangent to $a$ by the reflection $\sigma_b$ from $b$ 
is again a geodesic tangent to $a$. 
\end{definition}

\begin{remark} It is well-known that if $a$, $b$  are two confocal ellipses 
(ellipsoids) in Euclidean space, then the smaller one is a caustic for the bigger one. 
In the plane this is the classical 
Proclus--Poncelet theorem. In higher dimensions this theorem is due to Jacobi, 
see \cite[p.80]{st}. Similar statement holds in any space form, see, e.g., 
\cite[theorem 3]{veselov2}.
\end{remark}

We will deduce Theorem \ref{tab2} from 
the following theorem, which implies that {\it in every space form 
only quadrics  have caustics, and 
the caustics of each quadric $S$ are exactly the quadrics confocal to $S$.}

\begin{theorem} \label{berg} Let $d\geq3$, and let $\Sigma$ be a $d$-dimensional 
simply connected space form. Let 
$S,U\subset\Sigma$ be germs of  $C^2$-smooth hypersurfaces at points 
$B$ and $A\neq B$  respectively  with non-degenerate 
second fundamental forms. Let the geodesic $AB$ be tangent to 
$U$ at $A$ and transversal to $S$ at $B$.  Let $C\in\Sigma\setminus\{ B\}$, and let 
a vector tangent to the geodesic $AB$ at $B$ be reflected from the hyperplane $T_BS$ 
to a tangent vector to the geodesic $BC$. 
Let there exist a germ of $C^2$-smooth hypersurface 
 $V$ at $C$  tangent to $BC$ at $C$ such that  each geodesic close to $AB$ and 
 tangent to $U$ be reflected from the hypersurface $S$ to a geodesic tangent to $V$. 
Then $S$ is a piece of a quadric $b$, and 
$U$, $V$ are pieces of one and the same quadric confocal to $b$. 
\end{theorem}

\begin{remark} In the case, when $\Sigma=\rr^d$, Theorem \ref{berg} was proved 
by Marcel Berger \cite{berger}. 
\end{remark}

\section{Caustics of hypersurfaces in space forms. Proof of Theorem \ref{berg}}

The proof  of Theorem \ref{berg} for space forms essentially follows 
 Berger's proof for the Euclidean case given in \cite{berger}. In Subsection 2.1 we first prove that 
 the hypersurfaces $U$ and $V$ are pieces of the same quadric denoted by 
 $U$. Then in Subsection 2.2 we show that $S$ is a quadric confocal to $U$, using the 
 fact that it is an integral hypersurface of a  finite-valued hyperplane 
 distribution: the field of symmetry hyperplanes in $T_x\Sigma$, $x\in\Sigma$,    
 for the geodesic cones $K_x$ circumscribed about the quadric $U$ with vertex at $x$. 

\subsection{The hypersurfaces $U$ and $V$ and circumscribed cones} 
\begin{theorem} \label{th21} In the conditions of Theorem \ref{berg} the hypersurfaces 
$U$ and $V$ are pieces of one and the same quadric.
\end{theorem}
Theorem \ref{th21} is proved below following \cite{berger}. As in loc. cit.,  
we  first prove that for every $y\in S$ 
the geodesic cone with vertex $y$ tangent to $U$ is a quadratic cone tangent to both 
$U$ and $V$ (Lemma \ref{l1}). Afterwards we apply a result from \cite{berger} 
(stated below as Lemma \ref{l2} and proved in loc. cit. via arguments  using projective duality), showing 
that if the latter statement holds, then $U$ and $V$ lie in the same quadric. 

Let $\pi:\Sigma\to\rp^d$ denote the restriction to $\Sigma$ of 
 the tautological projection $\rr^{d+1}\setminus\{0\}\to\rp^d$. It is a diffeomorphism 
 onto the image $\pi(\Sigma)$ in non-spherical cases and a degree two covering over 
 $\rp^d$ in the spherical case.  
 Let $g$ denote the metric on $\pi(\Sigma)$ that is the (well-defined)  
 pushforward of the space form metric. 
 Note that the geodesics (completely geodesic subspaces) 
 for the metric $g$ are the intersections of  projective lines (respectively, projective 
 subspaces) with $\pi(\Sigma)$. In order to reduce the proof to the Euclidean case 
 treated in \cite{berger}, we use the following property of the metric $g$. 
 
 \begin{proposition} \label{pmet} For every point $y\in \pi(\Sigma)$ there exist an  
affine chart $\rr^d\subset\rp^d$ centered at $y$ and a Euclidean metric on $\rr^d$ 
(compatible with the affine structure) that has the same 1-jet at $y$, as the metric $g$. 
\end{proposition}
\begin{proof} Without loss of generality we consider that $y=(1:0:\dots:0)$:  
the isometry group of the space form $\Sigma$ 
acts transitively, and the projection $\pi$ conjugates its action on $\Sigma$ with 
its action on $\rp^d$ by projective transformations (since the isometry group is 
a subgroup in $GL_{d+1}(\rr)$). Thus,  in the standard affine chart $\{ x_0=1\}$ the point 
$y$ is the origin. The metric $g$ 
 is invariant under the orthogonal transformations 
of the affine chart, since the metric of the space form is invariant under the rotations 
around the $x_0$-axis. The metric $g$ on $T_y\rr^d$ coincides with the standard 
Euclidean metric of the chart $\rr^d$, by definition. The two last statements together 
imply that the 1-jets of both metrics at $y=0$ coincide. This proves the proposition.
\end{proof}

\begin{corollary} \label{csfund} Let $U\subset\Sigma$ be a germ of hypersurface with 
non-degenerate 
second fundamental form. Then its projection $\pi(U)$ has non-degenerate second 
fundamental form in any affine chart $\rr^d$ with respect to the standard Euclidean metric. 
\end{corollary}

\begin{proof} The corollary follows from Proposition \ref{pmet} and 
invariance of the property of having non-degenerate second fundamental form under 
projective transformations. Indeed, 
each germ of projective hypersurface is tangent to a unique quadric  
with order 3; non-degeneracy of the second fundamental form is equivalent to regularity 
of the quadric, and the space of regular quadrics is invariant under projective transformations. 
\end{proof} 

In what follows in the present subsection 
we identify the hypersurfaces $S$, $U$, $V$ and their points with 
their projection images: for simplicity the projection images $\pi(S)$, $\pi(U)$, $\pi(B)$ etc. 
will be denoted by the symbols $S$, $U$, $B$,...

\begin{lemma} \label{l1} Let $S,U,V\subset\rp^d$ be the tautological projection 
images of the same hypersurfaces in $\Sigma$, as in Theorem \ref{berg} (see the 
above paragraph).  For every $y\in S$ there exists a quadratic cone 
$K_y\subset\rp^d$ 
(i.e., given by the zero locus of a homogeneous quadratic polynomial) with vertex at 
$y$ that is tangent to both hypersurfaces $U$ and $V$. 
\end{lemma}
The proof of Lemma \ref{l1} given below follows  \cite[section 2]{berger}. 
 
Let $\sigma_g:(T\rp^d)|_S\to (T\rp^d)|_S$ denote the involution acting as 
 the symmetry of each space $T_y\rp^d$, $y\in S$, with respect to the 
 hyperplane $T_yS$ in the metric $g$. Its action on the 
 projectivized tangent spaces $\rp^{d-1}_y=\mathbb P(T_y\rp^d)$ induces its action on the space of projective lines in $\rp^d\supset S$ intersecting $S$ transversely and so that the intersection point 
 is unique: if $\ell$ intersects $S$ at a point $y$, then 
 $$\hat\ell:=\sigma_g(\ell)$$ 
 is the line through $y$ that is symmetric to $\ell$ in the above sense. 

For every $y\in \Sigma$ set 
$$M_y:=\text{ the space of projective lines through } y \text{ 
that are tangent to } U.$$

 It suffices to prove the statement of Lemma \ref{l1} for an arbitrary point $y\in S$ 
 satisfying the following statements.

 \begin{proposition} \label{pgener} (stated in \cite[pp. 110-111]{berger}) 
 There exists an open and dense subset of points $y\in S\subset\rp^d$ 
 for which there exists an open and dense subset $M^0_y\subset M_y$ of lines $\ell$ 
 satisfying the following statements: 

(i) the line $\ell$  is quadratically tangent to $U$, 
(i.e., $\ell$ is not an asymptotic direction  of the hypersurface $U$ at the tangency point); 

(ii) the line $\hat\ell=\sigma_g(\ell)$ is quadratically tangent to $V$ at a point, 
where the second fundamental form of the hypersurface $V$ is non-degenerate;  

(iii) the lines $\ell$ and $\hat\ell$ are transversal to $T_yS$ and 
 their above tangency points  with $U$ and $V$ are distinct from the point $y$. 
\end{proposition}

\begin{proof} Statement  (i)   holds for an open and dense subset of lines  $\ell\in M_y$,  
 since the second fundamental form of the hypersurface 
$U$ is non-degenerate (by assumptions and Corollary \ref{csfund}). Statement 
(iii) also holds for a generic $\ell\in M_y$, whenever $y\notin U\cup V$. 
Let us show that 
statement (ii) also holds generically. Indeed, let $y\in S$, $y\notin U\cup V$, 
and let $\ell$ be a line through $y$ satisfying assumption (i). 
Then the cone $K_y$ with vertex $y$ containing $\ell$ and 
circumscribed about $U$ is tangent to $U$ along a 
$n-2$-dimensional submanifold 
$\mct_U\subset U$. The correspondence sending a point $p\in \mct_U$ 
to the projective hyperplane tangent to $U$ at $p$ (i.e., to the projective 
hyperplane tangent to the cone along the line $yp$) 
is a local immersion to the space of hyperplanes through $y$. Or equivalently, 
the correspondence sending a line $L\subset K_y$ through $y$ to the projective 
hyperplane tangent to $K_y$ along $L$ is a local immersion. 
This follows from non-degeneracy of the second fundamental form of the hypersurface 
$U$. This implies the similar statement for the symmetric cone $\widehat{K_{y}}=\sigma_g(K_y)$ 
circumscribed about $V$: 
the correspondence sending each line $\widehat L\subset \widehat{K_{y}}$ through $y$ 
to the hyperplane tangent to $\widehat{K_{y}}$ along the line $\widehat L$ is a local immersion 
 to the space of hyperplanes through $y$. Suppose 
now that a line $\widehat L\subset\widehat{K_{y}}$ through $y$ 
 is quadratically tangent to $V$ at a point $q$. Then 
the above  immersivity statement for the symmetric cone together with 
quadraticity of tangency imply non-degeneracy of the second fundamental form 
of the hypersurface $V$ at the point $q$. It is clear that for a generic choice 
of the point $y$ and a line $L\subset K_y$ through $y$ the corresponding 
symmetric line $\widehat L=\sigma_g(L)$ is 
 quadratically tangent to $V$. This proves the proposition.
 \end{proof}

\begin{convention} \label{conv1} 
 Thus, in the proof of Lemma \ref{l1} 
 without loss of generality we consider that $y=B$, and there exists a line 
 $\ell$ through $B$ that is transversal to $S$ and satisfies statements (i)--(iii) 
 of Proposition \ref{pgener}. 
 Without loss of generality we consider that $A$ is the tangency point of the line 
 $\ell$ with $U$, and $C$ is the tangency point of the symmetric line $\hat\ell=
 \sigma_g(\ell)$ with 
 $V$;  $A,C\neq B$.  Fix an affine chart $\rr^d\subset\rp^d$ centered at $B$ and  
equipped with an Euclidean metric whose 1-jet at $B$ coincides with the 1-jet of the 
metric $g$ (Proposition \ref{pmet}). 
\end{convention}

Consider a smooth deformation $x(t)\in S$ of the point $B$, $x(0)=B$, and a 
smooth deformation $p(t)\in U$  of the point $A$, $p(0)=A$, such that the line 
$\ell(t)=x(t)p(t)$ is tangent to $U$ at $p(t)$. Then the line $\hat\ell(t)=\sigma_g(\ell(t))$ 
symmetric to $\ell(t)$ in the metric $g$ is tangent to the hypersurface $V$ at some 
point $q(t)$, $q(0)=C$, that depends  smoothly on the parameter $t$ (assumptions 
(i)--(iii)).  We will show that the property  
that every deformation $x(t)$ extends to a pair of deformations $p(t)$ and $q(t)$ as 
above implies that the cone $K_y$ tangent to both $U$ and $V$ is quadratic. To do this, consider the projective hyperplanes $\mcu$ and $\mcv$ through $B$ containing the 
lines $\ell(0)=BA$ and $\hat\ell(0)=BC$ respectively: 
$\mcu$ is tangent to $U$ at $A$, and $\mcv$ is tangent to $V$ at $C$. 
\begin{remark} Let $\mcu$ and $\mcv$ be as above. 
The tangent subspaces  $T_B\mcu, T_B\mcv\subset T_B\rp^d$ 
 are $\sigma_g$-symmetric, which 
follows from assumptions. In particular, the latter tangent spaces intersect on a 
codimension 2 subspace $H\subset T_B\rr^d$ lying in $T_BS$: 
\begin{equation}H=T_B\mcu\cap T_BS=T_B\mcv\cap T_BS.\label{newstar}
\end{equation}
For every deformations $x(t)$, $p(t)$, $q(t)$ as above one has 
\begin{equation}u=x'(0)\in T_BS, \  \ v=p'(0)\in T_A\mcu, \ w=q'(0)\in T_C\mcv.
\label{uvw}\end{equation}
This motivates the following definition 
\end{remark}
\begin{definition} Let $S$ be a germ of hypersurface at a point $B\in\rr^d\subset\rp^d$. 
Let $g$ be a positive definite scalar product on the bundle $T\rr^d|_S$. 
Let $\ell$ be a projective line 
through $B$ that is transversal to $T_BS$, 
and let $H\subset T_BS$ be a vector subspace of 
codimension one (codimension two in $T_B\rr^d$). 
Let $A\in\ell$, $C\in\hat\ell=\sigma_g(\ell)$, $A,C\neq B$. 
  Let $\mcu$ and $\mcv$ denote 
the projective hyperplanes through $B$ that are tangent to $H$ and such that 
$\ell\subset\mcu$, $\hat\ell\subset\mcv$. Let 
$$u\in T_BS, \ u\neq0, \ v\in T_A\mcu, \ w\in T_C\mcv.$$
We say that $(\ell, H,  u, A, v, C, w)$ is a {\it Berger tuple,} if there exist germs of 
$C^1$-smooth 
curves of points $x(t)\in S$, $p(t),q(t)\in\rp^d$, $x(0)=B$, $p(0)=A$, $q(0)=C$, 
 such that statements (\ref{uvw}) hold and for every small $t$ the lines 
$x(t)p(t)$ and $x(t)q(t)$ are $\sigma_g$-symmetric.
\end{definition} 

\begin{proposition} \label{bergeq} The property of being a Berger tuple depends only on 
the 1-jet of the metric $g$. Namely, let $S$ be a germ of hypersurface at a point 
$B\in\rr^d\subset\rp^d$. Let  $g_1$ and $g_2$ be two positive definite scalar products 
on the bundle $(T\rr^d)|_S$ that have the same 1-jet at $B$. Then any Berger 
tuple for the metric $g_1$ is a Berger tuple for the metric $g_2$ and vice versa.
\end{proposition}

\begin{proof}
The proposition follows from definition and smoothness of the dependence of the 
reflection $\sigma_g$ on the parameters of the metric $g$: if two metrics have the same 
1-jets at $B$, then the corresponding reflections acting in $T_y\rr^d$ differ by a quantity  
$o(y-B)$. 
\end{proof}

\begin{theorem} \label{thh} \cite[section 2]{berger}. Let $S$ be a 
germ of hypersurface at  a point $B\in \rr^d\subset\rp^d$. Consider the standard 
Euclidean metric on the affine chart $\rr^d$, and let $S$ have non-degenerate 
second fundamental form. Let  $\ell$ be a line through $B$ transversal to $T_BS$. 
Then there exist only a finite 
number $k\leq d-1$ of codimension one vector subspaces 
$H=H_1(\ell),\dots,H_k(\ell)\subset T_BS$ 
such that for every $u\in T_BS$, $u\neq0$ the triple $(\ell,H,u)$ 
extends to a Berger tuple $(\ell, H,  u, A, v, C, w)$ for the Euclidean metric. 
The number $k$ depends 
only on the quadric tangent to $S$ at $B$ with order 3 
(or equivalently, on the second fundamental form of the hypersurface $S$ at $B$). 
The subspaces $H_j(\ell)$ are uniquely 
determined by the line $\ell$ and the above quadric.
 \end{theorem}

\begin{proposition} \label{pmcd} \cite[p. 114]{berger}. 
In the conditions of Theorem \ref{thh} consider 
the tautological projection  $\pi_B:\rr^d\setminus\{ B\}\to\rp^{d-1}$ to the space of 
lines through $B$. For every line  $\ell$ through $B$ 
the corresponding projection $\pi_B(\ell\setminus\{ B\})\in\rp^{d-1}$ 
will be denoted by $[\ell]$. For every  $\ell$  transversal to $S$   let 
$\Delta_j(\ell)\subset T_B\rp^d=\rr^d$ denote the codimension 1 vector  
subspace spanned by $H_j(\ell)$ and $\ell$. Let 
$\wt\Delta_j([\ell])=\pi_B(\Delta_j(\ell)\setminus\{0\})\subset\rp^{d-1}$ 
 denote its tautological projection, which is a projective hyperplane through $[\ell]$. Set 
 $$\mcd_j([\ell]):=T_{[\ell]}\wt\Delta_j([\ell])\subset T_{[\ell]}\rp^{d-1}.$$
The 
subspaces $\mcd_1([\ell]),\dots,\mcd_k([\ell])\subset T_{[\ell]}\rp^{d-1}$ form a $k$-valued hyperplane distribution $\mcd$ 
on $\rp^{d-1}$, whose all integral surfaces are  quadrics. Moreover, let $\wt S$ be a 
quadric tangent to $S$ at $B$ with order 3: having the same second fundamental form 
at $B$. The $\pi_B$-preimages of the above quadrics in $\rp^{d-1}$ 
(i.e., the integral hypersurfaces) 
are cones with vertex at $B$ that are tangent to the  quadrics confocal to $\wt S$. 
\end{proposition}

\begin{proof} {\bf of  Lemma \ref{l1}.}  Let $K$ be the cone with vertex at $y=B$ 
circumscribed about the hypersurface $U$. Let  $A$ be a point of tangency 
of the cone $K$ with $U$. Set $\ell=BA$, $\hat\ell=\sigma_g(\ell)$. 
Let $C$ denote the point of tangency of the line $\hat\ell$ with $V$. (We 
consider that the assumptions of Convention \ref{conv1} hold.)   
Then for every germ of smooth curve $x(t)\subset S$, $x(0)=B$, there exist curves 
$p(t)\subset U$ and $q(t)\subset V$, $p(0)=A$, $q(0)=C$, 
such that the lines $x(t)p(t)$ and $x(t)q(t)$ 
are tangent to $U$ and $V$ at $p(t)$ and $q(t)$ respectively and $\sigma_g$-symmetric. 
Let $\mcu,\mcv\subset\rp^d$ be the previously defined projective hyperplanes  
through $B$ tangent to $U$ and $V$ at $A$ and $C$ respectively, and 
$H=T_B\mcu\cap T_B\mcv\subset T_BS$, see  (\ref{newstar}). 
The tuple $(\ell, H, x'(0), A,  p'(0), C, q'(0))$ is a Berger tuple for the 
metric $g$, by definition. 
Therefore, it is also a Berger tuple for the Euclidean metric as well 
(Proposition \ref{bergeq}). This together with Theorem \ref{thh} implies that 
$H=H_j(\ell)$ for some $j$. 
The  cone $K$ is tangent along the line $\ell$ to the hyperplane generated by $\ell$ 
and $H=H_j$, by definition. Therefore, the tautological projection 
$\wt K=\pi_B(K\setminus\{ B\})\subset\rp^{d-1}$   is tangent at $[\ell]=\pi_B(\ell)$ to the  
corresponding hyperplane $\mcd_j([\ell])$  from Proposition \ref{pmcd}. Finally, 
$\wt K$ is an integral hypersurface of the multivalued hyperplane distribution $\mcd$ 
from Proposition \ref{pmcd}, and hence, lies in a quadric $\Gamma(U)$. 
The preimage $\pi_B^{-1}(\Gamma(U))$ is a quadratic 
cone $K_B$ with vertex $B$ that contains $K$ and 
is $\sigma_g$-symmetric, being a cone tangent to a quadric confocal to $\wt S$ 
(see Proposition \ref{pmcd}). 
Similarly, the punctured cone $\sigma_g(K)\setminus\{ B\}$ tangent to $V$ is projected to 
a quadric $\Gamma(V)$, and $\sigma_g(K)$ lies in a quadratic cone. The latter 
quadratic cone coincides with $K_B$, by symmetry. 
 This proves Lemma \ref{l1}.
 \end{proof}

\begin{lemma} \label{l2} \cite[section 3]{berger} Let $U$, $V$, $S$ be $C^2$-smooth 
germs of hypersurfaces in $\rp^d$ with non-degenerate second fundamental forms. 
Let  for every $x\in S$ there exist a quadratic cone $K_x$ with vertex at $x$ 
that is tangent to both $U$ and $V$. Then $U$ and $V$ are pieces of one 
and the same quadric. 
\end{lemma}

\begin{proof} {\bf of Theorem \ref{th21}.} For every $y\in S$ close enough to $B$ 
there exists a quadratic cone $K_y$ with vertex at $y$ circumscribed about both 
$U$ and $V$ (Lemma \ref{l1}). Applying this statement to an open and dense 
subset of points $y\in S$ satisfying genericity assumptions from Convention \ref{conv1}  together with Lemma \ref{l2} yield that  $U$ and $V$ are pieces of one and the same 
quadric. Theorem \ref{th21} is proved.
\end{proof}

\subsection{Symmetry hyperplanes of circumscribed cones and confocal quadrics}
Here we prove the following lemma and then deduce  Theorem \ref{berg} from it.
\begin{lemma} \label{l3} 
(A generalization of an anologous statement in \cite[p. 109]{berger}.) 
Let $\Sigma$ be a simply connected space form of dimension at least three. 
Let  $U\subset\Sigma$ be a quadric with non-degenerate second fundamental form. 
For every $y\in\Sigma\setminus U$ let $K_y$ denote  
the geodesic cone circumscribed about the quadric $U$ with the vertex at $y$ 
(i.e., the union of geodesics through $y$ that are tangent to $U$). We identify the cone 
$K_y$ with the cone $\wt K_y\subset T_y\Sigma$ of vectors tangent to the above geodesics via the exponential mapping $\exp:T_y\Sigma\to\Sigma$. 
Let $S\subset\Sigma$ be a germ of hypersurface at a  point $B\notin U$ 
with non-degenerate second 
fundamental form such that for every $y\in S$ the cone $\wt K_y$ is symmetric with 
respect to the hyperplane $T_yS$. Then $S$ is a quadric confocal to $U$.
\end{lemma}
In the proof of Lemma \ref{l3} we use the following lemma. To state it, let us recall that 
the orthogonal polarity in $\rr^{d+1}$ is the correspondence sending each vector 
subspace to its orthogonal complement with respect to the standard Euclidean scalar 
product. The orthogonal polarity in codimension one, 
which sends codimension one vector subspaces to their 
orthogonal lines, induces a projective duality $\rp^{d*}\to\rp^d$ sending hyperplanes 
to points. It sends each 
hypersurface $S\subset\rp^d$ to its dual $S^*$: the family of  points dual to the 
hyperplanes tangent to $S$.  

\begin{definition} Consider a scalar product $<Gx,x>$ on $\rr^{d+1}$ 
defining a space form. Orthogonality with respect to the latter scalar product will 
be called {\it $G$-orthogonality.} 
Let $V\subset\rr^{d+1}$ be a subspace that is {\it not isotropic:} 
this means that the restriction to $V$ of the scalar product $<Gx,x>$ is a non-degenerate 
quadratic form (or equivalently, that $V$ is not tangent to the light cone 
$\{<Gx,x>=0\}$). The 
{\it pseudo-symmetry} with respect to $V$ is the 
linear involution $I_V:\rr^{d+1}\to\rr^{d+1}$ 
that preserves the 
above scalar product and whose fixed point set coincides with $V$: it acts trivially on 
$V$ and as a central symmetry in the $G$-orthogonal subspace.
\end{definition}

\begin{lemma} \label{prsym} 
 Let $V\subset\rr^{d+1}$ be a non-isotropic vector subspace. 
Let $k<d+1$. 
Consider the action $I_{V,k}:G(k,d+1)\to G(k,d+1)$ of the 
pseudo-symmetry with respect to $V$  on the Grassmannian of $k$-subspaces. The orthogonal polarity $L\mapsto L^\perp$ 
conjugates the actions $I_{V,k}$ and\footnote{Everywhere below the orthogonality sign 
$\perp$ means orthogonality with respect to the standard Euclidean scalar product.} 
  $I_{V^\perp, d+1-k}$. 
\end{lemma}
\begin{proof} The lemma seems to be well-known to specialists. 
In three dimensions it follows from  \cite[formula (15), p.23]{bolotin2}, \cite[formula (3.12), p.140]{kozlov}.  Let us present its proof for completeness of presentation. As it is shown 
below, Lemma \ref{prsym} is implied by the two following propositions.

\begin{proposition} \label{proport} Let $G$ be a real symmetric 
$(d+1)\times(d+1)$-matrix such that $G^3=G$. 
 Let two non-isotropic subspaces $V, W\subset\rr^{d+1}$ of complementary  
dimensions be  $G$-orthogonal. Then their 
{\bf Euclidean} orthogonal complements $V^\perp$ and $W^\perp$ are also 
non-isotropic  and $G$-orthogonal. 
\end{proposition}
\begin{proof} The condition of the proposition implies that the restrictions 
of the linear operator $G$ to $V$ and $W$ have zero kernels and 
\begin{equation} GV=W^\perp, \ GW=V^\perp.\label{pperp}\end{equation}
Thus, to prove $G$-orthogonality of the latter subspaces, it suffices to show that 
$$<G^2v,Gw>=<G^3v,w>=0 \text{ for every } v\in V \text{ and } w\in W.$$
The first equality follows from symmetry of the matrix $G$. The second 
one follows from $G$-orthogonality of the subspaces $V$ and $W$ and 
 the equality $G^3=G$. The subspaces (\ref{pperp}) are non-isotropic, since the 
 restrictions to them of the scalar product $<Gx,x>$ are isomorphic to its restriction 
 to $V$ and $W$ via the operator $G$: 
 for every $v_1,v_2\in V$ one has 
 $<G(Gv_1),Gv_2>=<Gv_1,v_2>$, since $G^3=G$. 
 Proposition 
\ref{proport} is proved.
\end{proof}

\begin{proposition} \label{psuniq} Let $<Gx,x>$ be a scalar product on $\rr^{d+1}$ 
defining a space form.  Let  $k\in\{1,\dots,d\}$,   $V\subset\rr^{d+1}$ 
be a non-isotropic subspace, and let $W\subset\rr^{d+1}$ be its 
non-isotropic\footnote{In the non-Euclidean cases the $G$-orthogonal complenent $W$ 
to a non-isotropic subspace $V$ is automatically non-isotropic. In the Euclidean case, when the matrix $G$ is degenerate, we require that $W$ does not contain its kernel:  
the $x_0$-axis.} 
$G$-orthogonal completent.  
Let $N_k(V)\subset G(k,d+1)$ denote the subset of those vector $k$-subspaces  
in $\rr^{d+1}$ that are direct sums of some  subspaces $\ell_1\subset V$ and 
 $\ell_2\subset W$.
 The pseudo-symmetry $I_V$  acts trivially on $N_k(V)$ 
 and induces a non-trivial projective involution $\rp^d\to\rp^d$.  
 Vice versa, every non-trivial projective involution acting trivially on $N_k(V)$ 
 is the projectivization of the  pseudo-symmetry $I_V$. 
\end{proposition}
\begin{proof} The first statement of the proposition is obvious. Let us prove the 
second one. Let $F:\rr^{d+1}\to\rr^{d+1}$ be a linear transformation 
whose projectivization is a non-trivial involution acting trivially on $N_k(V)$. 
Without loss of generality we consider that $F^2=\pm Id$. 
For every  vector subspace $L\subset V$ of dimension between 1 and 
$k$ the transformation $F$ preserves the collection of the $k$-subspaces in $N_k(V)$ 
containing $L$. Their intersection being equal to $L$, $F$ preserves $L$. 
The same statement holds for $L\subset W$. Therefore, the restriction of the 
transformation $F$ to any of the subspaces $V$ and $W$ is a homothety. 
The coefficients of the homotheties on $V$ and $W$ are equal to $\pm1$, since 
$F^2=Id$ up to sign. The signs of the latter coefficients 
are opposite, since the projectivization of the transformation $F$ is non-trivial. 
Hence, $F=\pm I_V$. This proves the proposition.
\end{proof}

Let us now return to the proof of Lemma \ref{prsym}. The action of a linear automorphism  
$F:\rr^{d+1}\to\rr^{d+1}$ on  all the vector subspaces of all the dimensions 
is conjugated via the orthogonal polarity to the similar action of the inverse 
$(F^*)^{-1}$ to the conjugate 
operator $F^*$ (with respect to the Euclidean scalar product). In the case, when 
$F$ is an involution, so is $F^*=(F^*)^{-1}$. 

{\bf Claim.} {\it The conjugate operator $F=I_V^*$ acts trivially on $N_{d+1-k}(V^\perp)$.}

\begin{proof} The orthogonal polarity sends each 
$k$-subspace $\Pi=\ell_1\oplus\ell_2\in N_k(V)$, $\ell_1\subset V$, 
$\ell_2\subset W$, to the intersection of two subspaces $L_j=L_j(\Pi)=\ell_j^\perp$:  
\begin{equation}
L_1\supset V^\perp, \  L_2\supset W^\perp, \ \Pi^\perp=L_1\cap L_2,\label{l12incl}\end{equation}
$$\dim(\Pi^\perp)=\dim L_1+\dim L_2-(d+1)
=d+1-k.$$
The transformation $F$ fixes $\Pi^\perp$, by construction 
and since the pseudo-symmetry $I_V$ fixes $\Pi$ (Proposition \ref{psuniq}). 
The  intersection $\Pi^\perp$ 
is the direct sum of the subspaces $L_1\cap W^\perp$ and 
$L_2\cap V^\perp$, which follows from  the inclusions (\ref{l12incl}), and hence, 
lies in $N_{d+1-k}(V^\perp)$. Vice versa, 
each point in $N_{d+1-k}(V^\perp)$ can be represented as the intersection 
$\Pi^\perp$ of 
some subspaces $L_1$ and $L_2$ containing $V^\perp$ and $W^\perp$ respectively. 
Therefore, $F$  acts trivially on all of $N_{d+1-k}(V^\perp)$. The claim is proved.
\end{proof}

The operator $F=I_V^*$ is a projectively non-trivial involution, as is $I_V$. 
It coincides with $I_{V^\perp}$ up to sign, by the claim and Proposition 
\ref{psuniq}.  This together with the discussion preceding the claim 
implies the statement of Lemma \ref{prsym}.
\end{proof}

\begin{proof} {\bf of Lemma \ref{l3}.} Consider the tautological projection 
$\pi:\rr^{d+1}\setminus\{0\}\to\rp^d$, the images  
$\pi(S),\pi(U)\subset\rp^d$ and the hypersurfaces in $\rp^d$ projective-dual 
to them with respect to the orthogonal polarity. For simplicity the latter projective-dual 
hypersurfaces will be denoted by $S^*$ and $U^*$ respectively. Let 
$\wt S,\wt U, \wt{S^*}, \wt{U^*}\subset\rr^{d+1}$ denote the complete $\pi$-preimages 
in $\rr^{d+1}$ of the hypersurfaces 
$\pi(S)$, $\pi(U)$, $S^*$ and $U^*$ respectively: the cones in $\rp^{d+1}\setminus\{0\}$ defined by the 
latter hypersurfaces. Recall that $\pi(U)$ and $U^*$ are dual quadrics; thus one can write 
$$U^*=\{<Qx,x>=0\}, \ \ Q \text{ is a real symmetric } (d+1)\times(d+1)-\text{matrix}.$$
For every $y\in S$ let $\mct_y S\subset\rp^d$ denote the projective hyperplane tangent to 
$S$ at $y$. Define the following vector subspaces  in $\rr^{d+1}$:
$$\Pi_y:=\pi^{-1}(\mct_yS)\cup\{0\}\subset\rr^{d+1}, \ L_y:=\Pi_y^\perp,$$
$$V_y:=\text{ the one-dimensional  subspace } \pi^{-1}(\pi(y))\cup\{0\}\subset\Pi_y, \  \ W_y:=V_y^{\perp}.$$
The subspaces $\Pi_y$, 
$L_y$, $V_y$, $W_y$ are non-isotropic, by construction and Proposition \ref{proport}. 

{\bf Claim 1.} {\it The quadric $U^*$ is regular, i.e., the matrix $Q$ is non-degenerate. 
The hyperplane section $\wt{U^*}\cap W_y$ is invariant 
under the pseudo-symmetry with   respect to the one-dimensional 
vector subspace $L_y\subset W_y$}

\begin{proof} The first statement (non-degeneracy) follows 
from non-degeneracy of the second fundamental form of the quadric $U$. The inclusion 
$L_y\subset W_y$ follows from definition. Recall that 
the  cone $\wt K_y$ is symmetric with respect to the hyperplane $T_yS$, i.e.,   
 the  preimage $\pi^{-1}(K_y)$ is pseudo-symmetric with respect to 
$\Pi_y$, by assumption. The latter statement is equivalent to the second statement 
of the claim, by duality and Lemma  
\ref{prsym}. 
\end{proof} 

The restriction to $W_y$ of the scalar product $<Gx,x>$ is non-degenerate 
(non-isotropicity), and there exist $d$ values $\la=\la_1(y),\dots,\la_{d}(y)$ (taken 
with multiplicity, some of them may coincide) such that the restriction to 
$W_y$ of the scalar product $<(Q-\la G)x,x>$ is degenerate. Thus, the $d$-dimensional 
vector subspace $W_y$ is the $G$-orthogonal direct sum of kernels of the 
scalar products $<(Q-\la_j(y)G)x,x>|_{W_y}$. 

{\bf Claim 2.}  {\it For every $y\in S$ the pseudo-symmetry line 
$L_y$ lies in the kernel of some of the 
scalar products $<(Q-\la_j(y)G)x,x>|_{W_y}$.}

\begin{proof} 
The scalar product $<Qx,x>|_{W_y}$ is invariant under the pseudo-symmetry 
with respect to the line $L_y$. Indeed, the latter pseudo-symmetry is an involution 
preserving the zero locus (light cone) $\wt{U^*}\cap W_y=\{<Qx,x>=0\}\cap W_y$ 
(Claim 1), 
and hence, it preserves the above scalar product up to sign. Let us show that 
the sign is also preserved. For an open and dense subset of points $y\in S$ 
one has  $<Qx,x>\neq 0$ on $L_y\setminus\{0\}$: 
equivalently (via duality), the tangent hyperplane $T_yS$ is not tangent to $U$. 
Indeed, the latter statement holds for an open and dense subset of points 
$y\in S$, since $S\cap U=\emptyset$ and a (germ of) hypersurface is uniquely 
defined by the family of its tangent hyperplanes (well-definedness of the dual 
hypersurface). Thus, for the above $y$ the pseudo-symmetry fixes the 
non-zero quadratic form $<Qx,x>|_{L_y}$, 
since the points of the line $L_y$ are fixed. This together with the above discussion 
 implies that the above-mentioned sign, and hence  the scalar 
product $<Qx,x>|_{W_y}$ are preserved for all $y\in S$. 

For every $\la_j(y)$ the  kernel of the form $<(Q-\la_j(y))x,x>|_{W_y}$ 
 is invariant under the above 
pseudo-symmetry, by  invariance of the scalar products $<Qx,x>$ and $<Gx,x>$. 
This is possible only in the case, when the pseudo-symmetry line $L_y$ lies in some 
of the kernels, which form an orthogonal direct sum decomposition of the subspace 
$W_y$. This proves Claim 2.
\end{proof}

\begin{remark} \label{reml} The subspace $W_y$ and hence, the corresponding 
kernels from Claim 2 depend only on $y$ and are well-defined for all $y\in\Sigma$. 
\end{remark}

Due to Claim 2, the following two cases are possible.

Case 1: for an open and dense subset $S_0$ 
of points $y\in S$ the line $L_y$ coincides with 
a one-dimensional kernel corresponding to a simple eigenvalue $\la_j(y)$. Let us show 
that in this case $S$ lies in a quadric confocal to $U$. Indeed then there 
exist a neighborhood 
$Y=Y(B)\subset\Sigma$ of the base point $B$ of the hypersurface $S$ 
and an open and dense subset $Y_0\subset Y$ containing $S_0$ 
such that the correspondence $y\mapsto L_y$ extends to a family of lines depending 
 analytically on $y\in Y_0$. This implies that the corresponding 
hyperplanes $\Pi_y:=L_y^\perp$ also depend analytically on $y$ and thus, induce 
a  field of hyperplanes $T=T(y)=\Pi_y\cap T_y\Sigma$ on $Y_0$. The hypersurface 
$S_0$ is its integral hypersurface.

Subcase 1.1): $U$ is a generic quadric. Then for a generic point $y\in\Sigma$ 
(here "generic" means "outside an algebraic subset") 

- there are exactly  $d$ quadrics through $y$ confocal to 
$U$, and any two of them are orthogonal at $y$;

- the corresponding eigenvalues $\la_j(y)$ are simple and 
the corresponding $d$ kernels in $W_y$ are one-dimensional.

Recall that the tangent hyperplanes at $y$ of the above confocal quadrics are 
symmetry hyperplanes for the cone $K_y$, since $U$ is a caustic for its 
confocal quadrics. Therefore, the orthogonal polarity $\Pi_y\mapsto L_y$ 
 induces a one-to-one correspondence between the above tangent hyperplanes 
 and kernels. 
This implies that for every $y\in Y_0$ the integral hypersurface of the hyperplane field 
$T$ through $y$ is a confocal quadric to $U$. Passing to limit, as $y$ tends to a point 
of the integral hypersurface $S$, we get that $S$ is a confocal quadric as well.

Subcase 1.2): $U$ is a general regular quadric. Then it is a limit of generic quadrics 
$U_n$ in the above sense. For each $U_n$ the integral hypersurfaces of the 
corresponding above hyperplane field $T_n$ are quadrics confocal to $U_n$. 
Passing to limit, as $n\to\infty$, we get the same statement for the hyperplane 
field $T$ associated to $U$. Hence, $S$ is a quadric confocal to $U$.

Case 2: there exists an open subset of points $y\in S$ for which $L_y$ lies in 
at least two-dimensional kernel of the form $<(Q-\la_j(y))x,x>|_{W_y}$ 
corresponding to a multiple eigenvalue $\la_j(y)$. 
In this case the latter kernel contains at least two linearly independent vectors 
$w_1,w_2\in W_y$, and by definition, both of them are orthogonal to 
the hyperplane $W_y$ with 
respect to the scalar product $<(Q-\la G)x,x>$, $\la=\la_j(y)$. Hence, their 
appropriate non-zero 
linear combination   $w=a_1w_1+a_2w_2$ is orthogonal to the whole ambient space 
$\rr^{d+1}$ with respect to the same scalar product. Therefore, 
$w$ lies in the kernel of the same scalar product taken  on all of $\rr^{d+1}$, 
and thus, $\lambda$ is such that the matrix $Q-\la G$ is degenerate: then we'll call 
such a $\la$ a {\it global eigenvalue}. 
The number of global  eigenvalues $\la$ is at most $d+1$, and all of them are independent 
on $y$. Finally, there exist a global eigenvalue $\lambda$ and an open subset $S_0\subset S$ such  that for every $y\in S_0$ one has $<(Q-\lambda)x,x>\equiv0$ on 
$L_y$, since $L_y$ lies in the kernel of the restriction to $W_y$ of the scalar product 
$<(Q-\lambda)x,x>$. 
Thus, for  $y\in S_0$ the projections $p(y)=\pi(L_y\setminus\{0\})\in\rp^d$ lie a degenerate quadric $\Gamma\subset\rp^d$ defined by the equation $<(Q-\lambda)x,x>=0$. The points $p(y)$ form the  dual hypersurface
$S_0^*$, by definition. Hence, $S_0^*$ lies in a degenerate quadric $\Gamma$. 
This contradicts non-degeneracy of the second 
fundamental form of 
the hypersurface $S$. Hence, the case under consideration is impossible. 
Lemma \ref{l3} is proved.
\end{proof}

\begin{proof} {\bf of Theorem \ref{berg}.} The hypersurfaces $U$ and $V$   lie in the same quadric in 
$\Sigma$, which will be now denoted by $U$ (Theorem \ref{th21}). The 
quadric $U$ is a caustic for the hypersurface $S$: for every $y\in S$ 
the cone of geodesics through $y$ that are tangent to $U$ is symmetric with 
respect to the hyperplane tangent to $T_yS$. Therefore, $S$ is a quadric confocal 
to $U$, by Lemma \ref{l3}. 
This proves Theorem \ref{berg}. 
\end{proof}

\section{Commuting billiards and caustics: proof of Theorem \ref{tab2}}

\begin{proposition} \label{prc} Let $\Sigma$ be a space form of constant curvature of 
dimension  $d\geq2$. Let two  nested  strictly convex  $C^2$-smooth closed 
hypersurfaces  $a,b\subset\Sigma$, $a\Subset\Omega_b$ (see the notations 
at the beginning of the paper) 
 be such that the corresponding billiard transformations $\sigma_a$ and 
 $\sigma_b$ commute. Then $a$ is a caustic for the hypersurface $b$. 
 \end{proposition}
 
 \begin{proof} Let $\Pi_a$ denote the open subset of geodesics in $\Sigma$ 
 that are disjoint 
 from the hypersurface $a$. Its boundary $\partial\Pi_a$ consists of those geodesics 
 that are tangent to $a$. 
A geodesic $L$ is fixed by $\sigma_a$, if and only if $L\in\overline\Pi_a$, i.e., $L$ is either 
disjoint from $a$, or tangent to $a$. In this case  
$\sigma_b\sigma_a(L)=\sigma_b(L)=\sigma_a\sigma_b(L)$, and thus,  $\sigma_b(L)$ is a fixed point of the
transformation $\sigma_a$. This implies that $\sigma_b(\overline\Pi_a)\subset\overline\Pi_a$. The subset $\overline\Pi_a$ is invariant under two transformations 
acting on oriented geodesics: the reflection $\sigma_b$ and the transformation $J$ 
of the orientation change. The transformations $J$ and $J\circ\sigma_b$ are involutions. 
Hence, they are homeomorphisms of the whole space of oriented geodesics in $\Sigma$. 
Their restrictions to the common invariant subset $\overline\Pi_a$ should be  also a homeomorphism: an involution acting on a set is obviously always bijective. Therefore, 
each of them sends the boundary $\partial\Pi_a$ onto itself homeomorphically, and 
the same is true for their composition $\sigma_b=J\circ (J\circ\sigma_b)$: 
$\sigma_b(\partial\Pi_a)=\partial\Pi_a$. 
The latter equality means exactly that $a$ is a caustic for the hypersurface $b$. 
The proposition is proved.
\end{proof}

\begin{proof} {\bf of Theorems \ref{tab2} and \ref{tab}.} Let $a,b\subset\Sigma$ be two nested strictly convex $C^2$-smooth closed hypersurfaces in a space form $\Sigma$ with commuting billiard transformations, $a\Subset\Omega_b$, $\dim\Sigma\geq3$. Then 
$a$ is a caustic for the hypersurface $b$, by Proposition \ref{prc}. This means that 
for every points $B\in b$ and $A\in a$ such that the line $AB$  is tangent to $a$ at $A$ 
the image $\sigma_b(AB)$ of the  line $AB$ (oriented from $A$ to $B$) is a line through 
$B$ tangent to $a$. Recall that $a$ and $b$ are strictly convex, which implies that 
their second fundamental forms are sign-definite and thus, non-degenerate.  Therefore, for every $A$ and $B$ as above the germs at $A$ and $B$ 
of the hypersurfaces $U=a$ and $S=b$ respectively satisfy the conditions of 
Theorem \ref{berg}, with $V$ being the germ of the hypersurface $a$ at its point $D$ 
of tangency with the line $\sigma_b(AB)$. Hence, for every $A$ and $B$ as above 
the germ $(S,B)$ lies in a quadric, and the germs $(U,A)$, $(V,D)$ 
 lie in one and the same quadric confocal to $S$. 
This implies that $b$ is a quadric, and $a$ is a quadric confocal to $b$. Theorems  
\ref{tab2}, and  \ref{tab} are proved.
\end{proof}

\section{A tangential local version of Theorem \ref{tab2}}

\begin{theorem} \label{tabl} Let $d\geq3$. Let 
$(U,A)$, $(S,B)$, $(V,D)$ be germs of $C^2$-smooth hypersurfaces in a $n$-dimensional 
space form $\Sigma$ at points $A$, $B$ and 
$D$. Let $B\neq A,D$,  and let $U$ and $S$ have non-degenerate second fundamental forms. 
For every $Z=U,S,V$ 
consider the action of the reflection $\sigma_Z$ on the 
oriented geodesics that intersect $Z$,  defined as at the beginning of the paper: 
we reflect the geodesic at its last intersection point with the hypersurface $Z$. 
Let  $L_0$ be a geodesic through $B$ transversal to  $S$ 
and quadratically tangent to $U$ at $A$ 
(we orient it from $A$ 
to $B$), and let its image $\sigma_S(L_0)$ be quadratically tangent to $V$ at $D$. 
Let $W$ be a small neighborhood of the 
geodesic $L_0$ in the space of oriented geodesics; 
in particular, each geodesic in $W$ intersects $S$ transversally. 
 Let $\Pi_W\subset W$ denote the subset of
 those geodesics that intersect  $U$. Let for every $L\in\Pi_W$ 
  the image $\sigma_S(L)$ intersect $V$; more precisely, we suppose that 
  the compositions 
$\sigma_S\circ\sigma_U$ and $\sigma_V\circ\sigma_S$ are well-defined on $\Pi_W$. 
Let the latter compositions be identically equal on $\Pi_W$. 
Then $S$ lies in a quadric $b$, and $U$, $V$ lie in one and the same quadric confocal to $b$. 
\end{theorem}

\begin{proof} Every geodesic $L$ tangent to $U$ and close enough to 
 $L_0$ lies in $\Pi_W$. Its image 
$\sigma_U(L)$ coincides with $L$ (by definition), and hence, 
$\sigma_S\circ\sigma_U(L)=\sigma_S(L)=
\sigma_V\circ\sigma_S(L)$. Thus, the geodesic $\sigma_S(L)$, which should 
 intersect $V$ by assumption, is fixed by $\sigma_V$. Hence, it is tangent to $V$ (at the last point 
 of its intersection with $V$). Finally, the germs of hypersurfaces $U$, $S$ and $V$ 
 satisfy the conditions of Theorem \ref{berg}. Therefore, $S$ lies in a quadric $b$, 
 and $U$, $V$ lie in one and the same quadric confocal to $b$, by Theorem \ref{berg}. 
 This proves Theorem \ref{tabl}.
 \end{proof} 
 
 \section{Open problems}
 
 The billiards in space forms are particular cases of the projective billiards introduced 
 in \cite{tabpr}.  The main results of the present paper 
 (Theorem \ref{berg} extending Berger's result on caustics \cite{berger},   
Theorem \ref{tab2} on commuting billiards) are proved for billiards in space forms. 
It would   interesting to extend them to projective billiards. 
 
 {\bf Problem 1} (appeared as a result of our discussion with Sergei Tabachnikov).  Let $S\subset \rr^d$, $d\geq3$ be a germ of hypersurface 
 at a point $B$  equipped 
 with a field $\La$ of  one-dimensional subspaces $\La_y\subset T_y\rr^d$, $y\in S$,  
 transversal to $S$. Consider the family of linear involutions 
 $\sigma_y:T_y\rr^d\to T_y\rr^d$, $y\in S$, that fix each point of the hyperplane $T_yS$ 
 and have $\La_y$ as an eigenline with eigenvalue $-1$. Let there exist 
 two germs of hypersurfaces $U$ and $V$ at points $A,C\neq B$ respectively 
 such that for every $y\in S$ each line through $y$ that is  tangent to $U$ is reflected 
 by $\sigma_y$ to a line tangent to $V$. (Thus defined action of the reflections $\sigma_y$ 
 on oriented lines transversal to $S$ is called the {\bf projective billiard transformation}, 
 and the pair $(S,\La)$ is called a {\bf projective billiard,} 
 see \cite{tabpr}.)    Is it true that then $U$ and $V$ lie in one 
 and the same quadric?
 
 {\bf Problem 2} (S.Tabachnikov). 
 Classify commuting nested pairs of projective billiards in $\rr^d$, $d\geq2$. 

\section{Acknowledgements}
I am grateful to Sergei Tabachnikov for attracting my attention to his Commuting Billiard 
Conjecture. I am grateful to him and to  Etienne Ghys, who informed me about Berger's 
theorem on caustics in higher-dimensional Euclidean spaces, 
for helpful discussions. I am grateful to the referee and to Sergei Tabachnikov for helpful remarks and for suggesting me 
to extend the results to billiards in spaces of constant curvature.

\end{document}